\documentclass[11pt]{article}

\usepackage{cite}
\usepackage{subfigure}
\usepackage{graphicx, color, graphpap}
\usepackage[normalem]{ulem}
\usepackage{soul}

\usepackage{amsmath,amssymb,amsthm}
\usepackage{ifthen}
\usepackage{mathrsfs}
\usepackage{mathtools}
\usepackage{enumitem}
\usepackage{lineno}
\usepackage{float}
\usepackage{fancyhdr}
\usepackage[us,12hr]{datetime} 
\usepackage{exscale}
\usepackage{tabularx}
\usepackage{latexsym}
\usepackage{fullpage}       
\usepackage{float}
\usepackage{verbatim}
\usepackage{multirow}
\usepackage{overpic}
\usepackage{comment} 
\usepackage{hyperref}
\usepackage{blkarray}
\usepackage{adjustbox}
\usepackage{tcolorbox}

\usepackage{pgfplots}
\usepackage{marvosym}
\usepackage[noabbrev]{cleveref}
\usepackage{tikz}
\usepackage{pgfplots}

\usepackage{algorithm}
\usepackage{algpseudocode}

\pgfplotsset{compat=1.11}
\usetikzlibrary{arrows, arrows.meta}
\usetikzlibrary{shapes.geometric, arrows.meta, positioning}

\tikzstyle{startstop} = [rectangle, rounded corners, minimum width=3.5cm, minimum height=1cm,text centered, draw=black, fill=blue!10]
\tikzstyle{process} = [rectangle, minimum width=4cm, minimum height=1cm, text centered, draw=black, fill=green!10]
\tikzstyle{io} = [trapezium, trapezium left angle=70, trapezium right angle=110, minimum width=3.5cm, minimum height=1cm, text centered, draw=black, fill=orange!15]
\tikzstyle{arrow} = [thick,->,>=stealth]

\definecolor{methD}{RGB}{230,120,50}
\definecolor{methDD}{RGB}{50,140,90}
\definecolor{methSDD}{RGB}{50,90,170}
\definecolor{methPSD}{RGB}{180,40,40}
\tikzset{
	visbox/.style={rounded corners=8pt, line width=1.2pt, align=center, inner sep=8pt},
	visnestouter/.style={visbox, minimum width=9.6cm, minimum height=4.4cm},
	visnest/.style={visbox, minimum width=7.6cm, minimum height=3.25cm},
	visnestmid/.style={visbox, minimum width=5.7cm, minimum height=2.2cm},
	visnestinn/.style={visbox, minimum width=3.4cm, minimum height=1.05cm,
		inner sep=3pt, font=\scriptsize\bfseries, align=center, text=methD},
	visnestpos/.style={anchor=south east},
	visnestinset/.style={xshift=-5pt, yshift=5pt},
}

\newcommand{\FW}{\mathbb{FW}}

\newcommand{\Snop}{\mathbb{S}_{+}^{n+1}}

\DeclareMathOperator{\ri}{ri}
\DeclareMathOperator{\range}{range}

\newcommand{\D}{\mathcal{D}}

\newcommand{\DD}{\mathcal{DD}}

\newcommand{\supp}{\text{supp}}

\newcommand{\A}{\mathcal{A}}

\newcommand{\tr}{\text{tr}}
\newcommand{\aff}{\text{aff}}

\newtheorem{prop}{Proposition}
\newtheorem{lem}{Lemma}[section]
\newtheorem{thm}{Theorem}[section]
\newtheorem{cor}{Corollary}[section]
\newtheorem{remark}{Remark}[section]

\newtheorem{assumption}{Assumption}[section]

\crefname{thm}{Theorem}{Theorems}
\Crefname{thm}{Theorem}{Theorems}
\crefname{problem}{Problem}{Theorems}
\Crefname{problem}{Problem}{Theorems}
\Crefname{assump}{Assumption}{Theorems}
\crefname{assump}{Assumption}{Theorems}
\crefname{assumption}{Assumption}{Assumptions}
\Crefname{assumption}{Assumption}{Assumptions}
\crefname{conjecture}{Conjecture}{Theorems}
\Crefname{conjecture}{Conjecture}{Theorems}
\crefname{prop}{Proposition}{Propositions}
\Crefname{prop}{Proposition}{Propositions}
\crefname{cor}{Corollary}{Corollaries}
\Crefname{cor}{Corollary}{Corollaries}
\crefname{lem}{Lemma}{Lemmas}
\Crefname{lem}{Lemma}{Lemmas}
\theoremstyle{definition}
\crefname{conj}{Conjecture}{Conjectures}
\Crefname{conj}{Conjecture}{Conjectures}
\crefname{remark}{Remark}{Remarks}
\Crefname{remark}{Remark}{Remarks}
\crefname{rmk}{Remark}{Remarks}
\Crefname{rmk}{Remark}{Remarks}
\crefname{example}{Example}{Examples}
\Crefname{example}{Example}{Examples}
\crefname{align}{}{}
\Crefname{align}{}{}
\crefname{equation}{}{}
\Crefname{equation}{}{}

\def\eqref#1{{\normalfont(\ref{#1})}}

\usepackage{authblk}

\author{Hao Hu\footnote{School of Mathematical and Statistical Sciences, Clemson University, Clemson, USA; Email: \url{hhu2@clemson.edu}; Research supported by the Air Force Office of Scientific Research under award number FA9550-23-1-0508.}
\qquad Mingming Xu\footnote{School of Mathematical and Statistical Sciences, Clemson University, Clemson, USA; Email: \url{mingmix@g.clemson.edu}; Website: \url{https://mindy-xu.github.io/}; Research supported by the Air Force Office of Scientific Research under award number FA9550-23-1-0508.}}

\begin{document}

\title{Face Structure in Partial Facial Reduction of Semidefinite Relaxations of Binary Programs}

\break
\date{}
\maketitle

\medskip


\begin{abstract}
Semidefinite relaxations of binary optimization problems often
fail Slater's condition.  Full facial reduction restores strict feasibility but
may require auxiliary problems as costly as the original relaxation.  We
analyze partial facial reduction based on tractable inner approximations of the
positive semidefinite cone: the nonnegative diagonal, diagonally dominant (DD),
and scaled diagonally dominant (SDD) cones.

For this class of SDP
relaxations, we prove that every terminal face
identified by SDD-partial FR is described by variables fixed to zero or one and
groups of variables forced to be equal.  The face admits a full-column-rank
\(\{0,1\}\) basis matrix with disjoint column supports, yielding a natural
reduced formulation that can preserve sparsity.  For polyhedral inner
approximations, we characterize the face exposed at each step through the
inequalities active on the associated outer relaxation.

We also establish sharp limitations.  We construct a family
with singularity degree two for which DD- and SDD-partial FR require the
maximum possible number of steps.  We further show that factor-width
restrictions can prevent detection of dense valid affine equalities.  These
results characterize both the structure recovered by tractable partial FR and
the cost of restricting the exposing-matrix search.
\end{abstract}

\section{Introduction}

Facial reduction removes degeneracy before a conic problem is
passed to a numerical solver.  For semidefinite programming (SDP), however,
each facial-reduction step generally requires
an auxiliary SDP to find an exposing matrix.  This issue is especially relevant when Slater's
condition fails, since the resulting lack of strict feasibility is associated
with ill-conditioning and weaker error bounds
\cite{sturm2000error,drusvyatskiy2017many,sremac2021error}.  Full facial reduction resolves
the degeneracy by identifying the minimal face containing the feasible set
\cite{borwein1981regularizing,borwein1981facial}, but
finding this face can still require semidefinite auxiliary problems, even with regularized formulations
\cite{lourenco2015solving}.

Partial facial reduction (partial FR)
reduces this cost
by restricting the exposing matrix to an inner approximation of the
positive semidefinite (PSD) cone \cite{permenter2018simplified}.
Permenter and Parrilo introduced this general preprocessing
framework, including reduced reformulations, dual recovery,
sparsity-preserving choices of inner approximations, and a software
implementation.  Building on their framework, we study
SDP relaxations
of binary optimization problems and characterize the structures recovered by
diagonal, DD, and SDD searches, as well as sharp limitations of these searches.
The choice of
inner approximation is therefore governed by two
requirements: the resulting auxiliary problem should be substantially cheaper than an SDP, but its exposing
matrices should still capture useful structure.  From the standpoint of
preprocessing cost, polyhedral and second-order-cone-representable inner
approximations are natural choices.  The nonnegative diagonal
and diagonally dominant (DD) cones lead to linear programming (LP) auxiliary
systems, while the scaled-diagonally dominant (SDD) cone, equivalently the
factor-width-two cone, leads to a second-order cone programming (SOCP)
auxiliary system
\cite{ahmadi2017optimization,ahmadi2019dsos,boman2005factor}.
By contrast, factor-width-$k$ cones with $k\geq3$ are represented using PSD blocks of order
$k$ and hence reintroduce semidefinite constraints into the auxiliary search
\cite{blekherman2022sparse,sootla2019block,wang2021polyhedral,
kirschner2024predictor,roig2022globally}.  Thus the diagonal, DD, and SDD cones are the
most relevant choices when partial FR is intended to serve as inexpensive
preprocessing.

Computational tractability alone does not determine whether such a restricted
search is useful.  A small inner approximation may fail to contain an exposing
matrix even when the current SDP remains facially degenerate, causing partial
FR to stop before reaching the minimal face.  Even when the
restricted search continues, partial FR may require substantially more
facial-reduction steps than full facial reduction.  We therefore ask three related questions:
What structure can LP- and SOCP-based searches recover from semidefinite relaxations of binary optimization problems?  When does an outer relaxation determine exactly the face exposed at a given step?  How much can restricting the search increase the number of facial-reduction steps?

We study SDP relaxations of a nonempty binary set
$P\subseteq\{0,1\}^n$.  We require
only that the relaxation contain the rank-one lift of every point in $P$ and
include the normalization and arrow constraints;
the relaxation may also include additional valid
linear constraints, including constraints derived from the
reformulation-linearization technique (RLT) or the
Sherali--Adams hierarchy
\cite{shor1987quadratic,lovasz1991cones,goemans1995improved,
laurent2003comparison,sherali1990hierarchy,sherali2013reformulation}.
Our structural characterization requires only the normalization and arrow
constraints and remains valid in the presence of additional valid linear
constraints.

Our main structural result characterizes the face identified when SDD-partial FR terminates: it is
defined by variables fixed to zero or one and groups of variables forced to be
equal.
This characterization directly yields a natural reduced formulation: fixed
variables can be eliminated, and each group of equal variables can be
represented by a single variable.
The resulting formulation retains a direct connection to the original binary variables and can preserve sparsity in the data matrices, whereas a generic matrix used to represent a face identified by full facial reduction may mix the original variables and produce dense reduced data matrices.
Although this characterization concerns the face identified at termination,
we also characterize the face identified at each partial facial-reduction step
for polyhedral inner approximations.  The associated LP outer relaxation
determines the next face, with explicit specializations to the nonnegative
diagonal and DD cones.

Beyond the structure of the identified faces, restricting the exposing-matrix
search can substantially increase the length of the reduction sequence.  We
construct a family for which full facial
reduction has singularity degree two, while DD- and SDD-partial FR require
$n$ steps, the maximum possible for a nonzero feasible SDP whose matrix
variable has order $n+1$.  Thus the long
reduction sequence is caused by restricting the exposing-matrix
search, rather than by a large singularity degree of the underlying SDP
formulation.  This contrasts with Sturm's classical family,
which has a matrix variable of the same order $n+1$ but singularity degree $n$
\cite{sturm2000error}.

Finally, we establish a limitation of partial FR based on the factor-width
hierarchy.  We construct a relaxation with $n$ binary variables for which
factor-width-$k$ partial FR fails, for every $k\leq n$, to detect a valid
affine equality involving all variables.  By contrast, affine facial reduction,
an alternative facial reduction algorithm for semidefinite relaxations of
binary programs \cite{hu2023affine}, detects this equality using affine-hull
information.

These results complement other recent approaches to degeneracy in SDP
relaxations of combinatorial problems.  Affine facial reduction
\cite{hu2023affine} constructs exposing matrices from affine-hull information,
whereas the primal approach in \cite{hu2025primal} exploits available feasible
solutions.  The singularity-degree results in \cite{hu2026shor} concern full
facial reduction for Shor relaxations under linear and quadratic descriptions
of binary sets.
The present paper instead provides a binary-specific
structural and worst-case analysis of partial FR based on tractable inner
approximations.
In particular,
our sharp step-count separation holds even though the underlying SDP has
singularity degree two.

\paragraph{Organization.}
\Cref{sec_prel} fixes notation and recalls SDP relaxations for binary programs,
the facial reduction algorithm, the diagonal, DD, and
factor-width cones, and partial FR via inner
approximations.  \Cref{sec:binary_relation_structure_selection} first
characterizes the structure of the face identified by SDD-partial FR and then
shows how the inequalities active throughout a polyhedral outer relaxation
determine the face exposed by diagonal- and DD-partial FR.
\Cref{sec:sharp_partial_fra_separation} establishes the sharp separation
between the number of steps required by DD- and SDD-partial FR and full
facial reduction.
\Cref{sec:factor_width_support_limitation} establishes the
factor-width-$k$ support limitation.

\section{Preliminaries}\label{sec_prel}

\subsection{Notation}
For a positive integer $p$, let $\mathbb{R}^p$ denote the $p$-dimensional
Euclidean space, $\mathbb{S}^p$ the space of $p\times p$ real symmetric
matrices, and $\mathbb{S}^p_+$ the positive semidefinite cone.  We use the
standard Euclidean inner product for vectors and the Frobenius inner product
$\langle X,Y\rangle:=\tr(XY)$ for symmetric matrices.

Unless a coordinate set is stated explicitly, $\mathbb{R}^N$ and
$\mathbb{S}^N$ are indexed by $\{1,\ldots,N\}$.  The lifted spaces
$\mathbb{R}^{n+1}$ and $\mathbb{S}^{n+1}$ are instead indexed by
$\{0,\ldots,n\}$.  Standard unit vectors are denoted by $e_i$, and
$\mathbf{1}$ denotes an all-ones vector; their dimensions are determined by
context.  For indices $i,j$ in the coordinate set of $\mathbb{S}^N$, define
\begin{equation}\label{eq:symmetric_basis_matrices}
E_{ii}:=e_ie_i^T,
\qquad
E_{ij}:=\frac{1}{2}(e_ie_j^T+e_je_i^T)\quad (i\neq j).
\end{equation}
Thus $\langle E_{ij},X\rangle=X_{ij}$ for every $X\in\mathbb{S}^N$.

For a vector $v$, its \emph{support} is
$\supp(v):=\{i\mid v_i\neq0\}$.  For a matrix $X$, $\range(X)$ and $\ker(X)$
denote its range and null space, respectively.  For a convex set $C$, $\ri(C)$
denotes its relative interior.  For a cone $K$, its dual cone is
$K^*:=\{z\mid \langle z,x\rangle\geq0\ \text{for all }x\in K\}$.
For a subset $S$ of an inner-product space, we write
$S^\perp:=\{z\mid \langle z,s\rangle=0\ \text{for all }s\in S\}$.
This convention also applies when $S$ is affine; equivalently,
$S^\perp=\operatorname{span}(S)^\perp$.
For a single vector or matrix $s$, the notation $s^\perp$ abbreviates
$\{s\}^\perp$.

\subsection{SDP relaxations for binary programs}
\label{subsec:sdp_relaxations_binary_programs}
Throughout, $P\subseteq\{0,1\}^n$ is assumed to be nonempty.
An \emph{SDP relaxation} of $P$ is an affine slice
$L\cap\mathbb{S}^{n+1}_+$, where $L\subseteq\mathbb{S}^{n+1}$ is an affine
subspace, such that
\[
\begin{pmatrix}1\\ x\end{pmatrix}
\begin{pmatrix}1\\ x\end{pmatrix}^T
\in L\cap\mathbb{S}^{n+1}_+
\qquad\text{for every }x\in P.
\]

\begin{assumption}[Normalization and arrow constraints]
\label{ass:binary_sdp_relaxation}
Every SDP relaxation $L\cap\mathbb{S}^{n+1}_+$ considered below includes the
normalization and \emph{arrow constraints}; that is, every $Y\in L$, and hence
every $Y\in L\cap\mathbb{S}^{n+1}_+$, satisfies
\begin{equation}\label{eq:binary_arrow_constraints}
Y_{00}=1,\qquad
Y_{ii}=Y_{0i},\qquad i=1,\ldots,n.
\end{equation}
The normalization constraint fixes the entry indexed by \(0\),
while each arrow constraint represents the binary identity
$x_i^2=x_i$ in the rank-one lift.
\end{assumption}

With the convention \eqref{eq:symmetric_basis_matrices}, we write the normalization and arrow constraints as
\begin{equation}\label{eq:arrow_matrix_defs}
\langle A_0,Y\rangle=1,\qquad
\langle A_i,Y\rangle=0\quad (i=1,\ldots,n),
\qquad
A_0 := E_{00}, \quad A_i := E_{ii}-E_{0i}.
\end{equation}

\subsection{Facial reduction algorithm (FRA)}

Facial reduction replaces the ambient PSD cone by successively smaller faces
that contain the feasible set.  To describe one facial-reduction step, consider
the SDP feasibility problem
\[
L\cap\mathbb{S}^N_+,
\qquad
L:=\{X\in\mathbb{S}^N\mid \A(X)=b\},
\]
where $L$ is an affine subspace, $\A:\mathbb{S}^N\to\mathbb{R}^m$ is linear and $\A^*$ is its adjoint.
Assume that $L\cap\mathbb{S}^N_+$ is nonempty.  If the problem fails
Slater's condition, a facial-reduction step
\cite{borwein1981regularizing,borwein1981facial} finds an exposing matrix $W$
in the set
\begin{equation}\label{eq:first_fr_auxiliary}
\bigl(L^\perp\cap\mathbb{S}^N_+\bigr)\setminus\{0\}
=
\left\{
\A^*(y) \in \mathbb{S}^N_+
\ \middle|\
y\in\mathbb{R}^m,\ b^Ty=0
\right\}\setminus\{0\}.
\end{equation}
We call \eqref{eq:first_fr_auxiliary} the
\emph{facial-reduction auxiliary system}.
It then replaces $\mathbb{S}^N_+$ by the proper face
$F:=\mathbb{S}^N_+\cap W^\perp$, which contains the feasible set.  Since
$W\succeq0$, if the columns of $V\in\mathbb{R}^{N\times r}$ form a basis of
$\ker(W)$, then
\[
F
=
\{X\succeq0\mid \range(X)\subseteq\ker(W)\}
=
\{VRV^T\mid R\in\mathbb{S}^r_+\}.
\]
Here $F$ is a face in the original matrix space $\mathbb{S}^N$, whereas
$\mathbb{S}^r_+$ is the smaller cone that parametrizes this face through
$X=VRV^T$.
Substituting $X=VRV^T$ therefore gives the equivalent facially reduced problem
\[
\widetilde L\cap\mathbb{S}^{r}_+,
\qquad
\widetilde L
:=\{R\in\mathbb{S}^{r}\mid VRV^T\in L\}.
\]
Equivalently, its constraint map is
$\widetilde\A(R):=\A(VRV^T)$, with adjoint
$\widetilde\A^*(y)=V^T\A^*(y)V$.

If this reduced problem still fails Slater's condition, the same procedure is
applied to $\widetilde L\cap\mathbb{S}^{r}_+$. Repeating these steps produces a chain of smaller PSD cones.
The procedure terminates when the current reduced problem satisfies Slater's
condition, at which point the corresponding face of $\mathbb{S}^N_+$ is the
minimal face containing the original feasible set.
The minimum number of facial-reduction steps required
to reach this minimal face is the \emph{singularity degree} of the SDP
feasibility problem.

Although every basis of $\ker(W)$ represents the same
face, the choice of \(V\) affects the structure of this reduced SDP.  A dense
basis can produce dense reduced constraint matrices even when the original
constraint matrices are sparse.

\subsection{Partial facial reduction via inner approximations}
\label{subsec:partial_fr_inner_approximations}
Solving the facial-reduction auxiliary system for a nonzero exposing matrix
can be computationally expensive.  Partial FR restricts the search to a
tractable inner approximation of the PSD cone.  Let
\(\mathcal{K}\subseteq\mathbb{S}^N_+\) be a closed convex cone.  For the feasible set
\(L\cap\mathbb{S}^N_+\), the first \(\mathcal K\)-partial FR auxiliary system
seeks a matrix \(W\) in
\begin{equation}\label{eq:partial_fr_search}
\bigl(L^\perp\cap\mathcal K\bigr)\setminus\{0\}
=
\left\{
\A^*(y)\in\mathcal K
\ \middle|\
y\in\mathbb{R}^m,\ b^Ty=0
\right\}\setminus\{0\}.
\end{equation}
Since \(\mathcal{K}\subseteq\mathbb{S}^N_+\), every solution of
\eqref{eq:partial_fr_search} also solves the facial-reduction auxiliary system
\eqref{eq:first_fr_auxiliary}.  Hence \(W\) exposes the proper face
$\mathbb{S}^N_+\cap W^\perp$, which contains the feasible
set.
Moreover, this inclusion implies
\(\mathbb{S}^N_+\subseteq\mathcal K^*\).  We therefore call
\(L\cap\mathcal K^*\) the \emph{outer relaxation} associated with
\(\mathcal K\).

After parametrizing the exposed face by a smaller PSD cone, we apply the
restricted auxiliary system to the reduced SDP and repeat.  At each iteration,
\(\mathcal K\) denotes the chosen inner approximation in the dimension of the
current PSD cone, and \(L\) denotes the affine subspace in the current reduced
formulation.  Whenever \(L^\perp\cap\mathcal K\neq\{0\}\), we choose
\[
W\in\ri(L^\perp\cap\mathcal K).
\]
We call this iterative procedure \emph{\(\mathcal K\)-partial FR}.

The procedure terminates when \(L^\perp\cap\mathcal K=\{0\}\) in the current
reduced formulation.  The resulting face, expressed in the original matrix
space, is the \emph{face identified by \(\mathcal K\)-partial FR}.  Choosing a
tractable cone such as a factor-width cone can reduce the cost of the
restricted search.

\begin{remark}[Representation dependence of partial FR]
\label{rem:partial_fra_representation}
By \eqref{eq:first_fr_auxiliary}, full facial reduction requires the exposing
matrix to be positive semidefinite.  Thus the facial-reduction auxiliary
system searches over
$L^\perp\cap\mathbb{S}^N_+$.  Suppose that the
current face is represented as
$F=\{VRV^T\mid R\in\mathbb S^r_+\}$, where $V$ has full column rank.  If
$\widehat V=VQ$ for a nonsingular matrix $Q$, then $\widehat V$ gives another
representation of the same face.  Since the PSD cone is invariant under
nonsingular congruence,
\[
V^T\A^*(y)V\succeq0
\quad\Longleftrightarrow\quad
\widehat V^T\A^*(y)\widehat V
=
Q^T\bigl(V^T\A^*(y)V\bigr)Q\succeq0.
\]
Thus the facial-reduction auxiliary system is invariant under
replacing $V$ by $\widehat V$.

In contrast, an inner approximation $\mathcal K\subseteq\mathbb{S}^r_+$ is
generally not invariant under such a change of basis:
\[
Z\in\mathcal K
\quad\not\Longleftrightarrow\quad
Q^TZQ\in\mathcal K.
\]
Consequently, $\mathcal K$-partial FR can depend on the matrix
$V$ used to represent the current face.  Whenever this choice matters below,
we specify $V$ explicitly.
\end{remark}

\Cref{fig:partial_fr_restricted_search} illustrates this restriction at the
level of the auxiliary search.

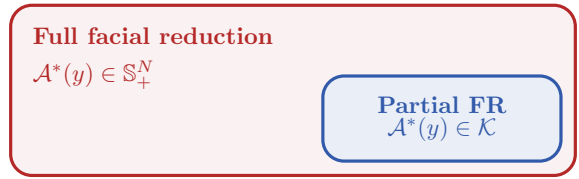
\begin{figure}[H]
\centering
	\begin{tikzpicture}[scale=.9, every node/.style={transform shape}]
		\node[visbox, draw=methPSD, fill=methPSD!6, minimum width=8.3cm, minimum height=2.5cm] (frbox) {};
		\node[font=\small\bfseries, text=methPSD, anchor=north west]
			at ([xshift=6pt,yshift=-6pt]frbox.north west) {Full facial reduction};
		\node[font=\small, text=methPSD, anchor=north west]
			at ([xshift=6pt,yshift=-7mm]frbox.north west) {$\A^{*}(y)\in\mathbb{S}^N_+$};
		\node[visbox, draw=methSDD, fill=methSDD!10, minimum width=3.5cm, minimum height=1.25cm,
			anchor=south east] (pfrbox) at ([xshift=-6pt,yshift=6pt]frbox.south east) {};
		\node[font=\small\bfseries, text=methSDD, align=center] at (pfrbox.center)
			{Partial FR\\[-1mm] $\A^{*}(y)\in \mathcal{K}$};
	\end{tikzpicture}
\caption{Partial FR restricts the auxiliary search from
$\mathbb{S}^N_+$ to a tractable cone $\mathcal{K}$.}
\label{fig:partial_fr_restricted_search}
\end{figure}

\subsection{Factor-width, DD, and SDD cones}
\label{subsec:factor_width_sdd_cones}
For $k\in\{1,\ldots,n\}$, the \emph{factor-width-$k$ cone} is
\begin{equation}\label{eq:fwk_def}
\FW_k^n
:=
\left\{
\sum_{t=1}^{q} g_tg_t^T
\;\middle|\;
q\geq1,\quad
g_t\in\mathbb R^n,\quad
|\supp(g_t)|\leq k\quad (t=1,\ldots,q)
\right\}.
\end{equation}
Thus \(\FW_k^n\subseteq\mathbb S^n_+\), and its dual consists of the matrices
in \(\mathbb S^n\) whose principal submatrices of order at most \(k\) are
positive semidefinite.

The first factor-width cone,
\(\D:=\FW_1^n\), is the nonnegative diagonal cone.  We also use the
diagonally dominant cone
\[
\DD
:=
\left\{
X\in\mathbb S^n
\ \middle|\
X_{ii}\geq\sum_{j\neq i}|X_{ij}|,\quad i=1,\ldots,n
\right\}.
\]
Both \(\D\) and \(\DD\) are polyhedral.  Their rank-one generators and the
corresponding defining inequalities of the dual cones are recorded in
\Cref{tab:polyhedral_cones}; these descriptions will be used in
\Cref{sec:outer_relaxation_characterizations}.
\begin{table}[ht]
\centering
\caption{Rank-one generators of $\D$ and $\DD$ and the corresponding dual
inequalities.}
\label{tab:polyhedral_cones}
\renewcommand{\arraystretch}{1.35}
\begin{tabularx}{\textwidth}{@{}cXX@{}}
\hline
\textbf{Cone} & \textbf{Rank-one generators} &
\textbf{Dual inequalities}\\
\hline
$\D$
& $e_ie_i^T$, $1\leq i\leq n$
& $Y_{ii}\geq0$, $1\leq i\leq n$\\
$\DD$
& $e_ie_i^T$, $1\leq i\leq n$;\newline
  $(e_i\pm e_j)(e_i\pm e_j)^T$, $1\leq i<j\leq n$
& $Y_{ii}\geq0$, $1\leq i\leq n$;\newline
  $Y_{ii}+Y_{jj}\pm2Y_{ij}\geq0$, $1\leq i<j\leq n$\\
\hline
\end{tabularx}
\end{table}

The second factor-width cone \(\FW_2^n\) is precisely the
\emph{scaled diagonally dominant (SDD) cone}: \(X\in\FW_2^n\) if and only if
there exists a positive diagonal matrix \(D\) such that \(DXD\in\DD\)
\cite{boman2005factor,ahmadi2019dsos}.  In particular,
\[
\D\subseteq\DD\subseteq\FW_2^n\subseteq\mathbb S^n_+.
\]
We will also use the following equivalent characterization.  For
\(X\in\mathbb S^n\), define its \emph{comparison matrix} by
\begin{equation}\label{eq:comparison_matrix_def}
C(X)_{ii}=X_{ii},\qquad C(X)_{ij}=-|X_{ij}|\quad (i\neq j).
\end{equation}
Then
\[
X\in\FW_2^n
\quad\Longleftrightarrow\quad
C(X)\succeq0.
\]
This characterization is standard; see \cite{boman2005factor}.

\section{Structure of faces identified by partial facial reduction}
\label{sec:binary_relation_structure_selection}

This section characterizes the faces identified by partial FR using the
diagonal, DD, and SDD cones.  We first show that SDD-partial FR identifies
faces described by variables fixed to zero or one and groups of variables
forced to be equal.  We then specialize this description to diagonal- and
DD-partial FR.  For these polyhedral cones, we also show how the associated
outer relaxation determines the face exposed at an individual step.

The hierarchy of these searches and the structures they can identify are
summarized in \Cref{fig:partial_fra_nested_settings}.  Here and below, $\D$
and $\DD$ denote the corresponding cones in the current matrix space.

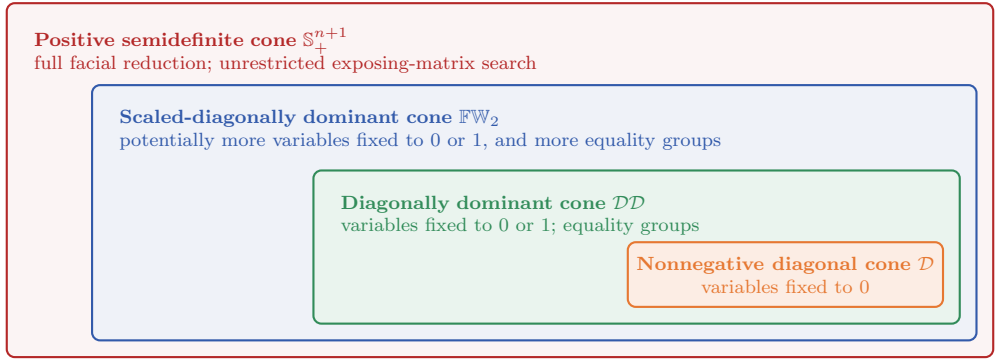
\begin{figure}[H]
\centering
\begin{tikzpicture}[
		scale=0.9,
		every node/.style={transform shape},
		nestbox/.style={rounded corners=3pt, line width=.8pt},
		nesttext/.style={anchor=north west, align=left, font=\scriptsize}
]
	\node[nestbox, draw=methPSD, fill=methPSD!5,
		minimum width=14.5cm, minimum height=5.2cm] (psd) {};
	\node[nesttext, text=methPSD, text width=13.4cm]
		at ([xshift=8pt,yshift=-7pt]psd.north west)
			{\textbf{Positive semidefinite cone} \(\mathbb{S}^{n+1}_+\)\\
		full facial reduction; unrestricted
		exposing-matrix
		search};

	\node[nestbox, draw=methSDD, fill=methSDD!8,
		minimum width=13.0cm, minimum height=3.75cm,
		anchor=south east] (sdd)
		at ([xshift=-7pt,yshift=7pt]psd.south east) {};
		\node[nesttext, text=methSDD, text width=12.0cm]
			at ([xshift=8pt,yshift=-7pt]sdd.north west)
				{\textbf{Scaled-diagonally dominant cone} \(\FW_2\)\\
				potentially more variables fixed to \(0\) or \(1\), and more equality groups};

	\node[nestbox, draw=methDD, fill=methDD!10,
		minimum width=9.5cm, minimum height=2.25cm,
		anchor=south east] (dd)
		at ([xshift=-7pt,yshift=7pt]sdd.south east) {};
	\node[nesttext, text=methDD, text width=8.0cm]
		at ([xshift=8pt,yshift=-7pt]dd.north west)
		{\textbf{Diagonally dominant cone} \(\DD\)\\
		variables fixed to \(0\) or \(1\); equality groups};

	\node[nestbox, draw=methD, fill=methD!13,
		minimum width=3.75cm, minimum height=.95cm,
		anchor=south east, text=methD, align=center, font=\scriptsize] (d)
		at ([xshift=-7pt,yshift=7pt]dd.south east)
		{\textbf{Nonnegative diagonal cone} \(\D\)\\
		variables fixed to \(0\)};
\end{tikzpicture}
\caption{Restricted facial-reduction searches and the structures they can
identify.  Enlarging the search cone from \(\D\) to \(\DD\) to
\(\FW_2^{n+1}\) can identify additional variables fixed to zero or one and
additional groups of variables forced to be equal; searching over
\(\mathbb S^{n+1}_+\) recovers unrestricted facial reduction.}
\label{fig:partial_fra_nested_settings}
\end{figure}

\subsection{Structure of faces identified by SDD-partial FR}
\label{sec:structured_partial_fra}
The SDD cone is the factor-width-two cone defined in
\Cref{subsec:factor_width_sdd_cones}.  Let $W\in\FW_2^{n+1}$ be nonzero.
By \eqref{eq:fwk_def}, \(W\) admits a decomposition
\[
W=\sum_{\ell=1}^p g_\ell g_\ell^T,
\qquad g_\ell\neq0,\quad
|\supp(g_\ell)|\leq2,\quad \ell=1,\ldots,p.
\]
For $Y\succeq0$, each $g_\ell^TYg_\ell$ is nonnegative, and hence
$\langle W,Y\rangle=0$ if and only if $Yg_\ell=0$ for every
$\ell=1,\ldots,p$.  Therefore
\begin{equation}\label{eq:fw2_exposed_face_two_coordinate}
\mathbb S^{n+1}_+\cap W^\perp
=
\left\{Y\succeq0\ \middle|\
Yg_\ell=0\ \text{ for }\ell=1,\ldots,p\right\}.
\end{equation}
Thus an arbitrary SDD exposing matrix imposes kernel
equations supported on at most two coordinates, with unrestricted nonzero
coefficients.

Suppose now that $L\cap\Snop$ is an SDP relaxation of
$P\subseteq\{0,1\}^{n}$ satisfying \Cref{ass:binary_sdp_relaxation}.  The
normalization and arrow constraints sharply restrict the
coefficients that can persist at termination.  We now characterize the
resulting terminal faces.

The argument begins with the following standard
support-graph description of the null space of a factor-width-two matrix.

\begin{prop}[Disjoint column supports]
\label{prop_sdd_disjoint_support_basis}
Let \(W\in\FW_2^{n+1}\).
Then $\ker(W)$ has a basis matrix $V$ whose columns have pairwise disjoint
supports, and such a basis can be computed in sparse form by a graph search
from any factor-width-two decomposition of \(W\).
\end{prop}
\begin{proof}
For completeness, we give the standard support-graph argument for
factor-width-two decompositions; see, e.g.,
\cite{boman2005factor,permenter2018simplified}.
Choose a factor-width-two decomposition
\(W=\sum_{t=1}^p g_tg_t^T\), where \(g_t\neq0\) and
\(|\supp(g_t)|\leq2\) for \(t=1,\ldots,p\).
For any vector \(z\),
\[
z^TWz=\sum_{t=1}^{p}(g_t^Tz)^2.
\]
Thus \(z\in\ker(W)\) if and only if \(g_t^Tz=0\) for every \(t\).
We now build a multigraph \(G\) on the vertices \(\{0,\ldots,n\}\), together
with a set of marked vertices.  A marked vertex means that the corresponding
coordinate is forced to be zero.

For each \(t\), if \(\supp(g_t)=\{i\}\), then \(g_t^Tz=0\) imposes
\(z_i=0\); mark vertex \(i\).  If \(\supp(g_t)=\{i,j\}\), write
\(g_t=\alpha e_i+\beta e_j\), where \(\alpha\beta\neq0\), and add an edge
\(\{i,j\}\) to \(G\), allowing parallel edges from different vectors \(g_t\).
The equation \(g_t^Tz=0\) gives the two equivalent propagation rules
$z_j=-(\alpha/\beta)z_i$ and $z_i=-(\beta/\alpha)z_j$.

Suppose first that the graph is connected.  We distinguish three cases.  If it
contains a marked vertex \(i\), then \(z_i=0\) for every \(z\in\ker(W)\).  Since
every other vertex is connected to \(i\) by a path, the propagation rules force
\(z_j=0\) for all \(j\).  Thus \(\ker(W)=\{0\}\) in this case.

Next suppose that no vertex is marked.  Choose an arbitrary root vertex \(r\),
set \(z_r=1\), and label \(r\).  While an unlabeled vertex remains,
connectedness provides an unlabeled vertex \(u\) adjacent to at least one
labeled vertex.  For every edge between \(u\) and the labeled vertices adjacent
to \(u\), including all parallel edges, use the propagation rules to obtain a
proposed value for \(z_u\).  If all proposed values agree, label \(u\) with
their unique common value.  If they disagree, the edge equations force
\(z_r=0\), and connectedness then forces \(z=0\); terminate.  Repeat until a
conflict occurs or every vertex is labeled.

If all vertices are labeled without conflict, the resulting nonzero vector
\(z\) satisfies every edge equation, and the propagation rules show that every
vector in \(\ker(W)\) is a scalar multiple of \(z\).

If the graph is disconnected, apply the same labeling procedure to each
connected component.  Components containing a marked vertex or producing a
conflict contribute nothing to the null space.
Each remaining component
contributes one vector, which we extend by zeros outside that component.  These
extended vectors have pairwise disjoint supports and span \(\ker(W)\).  Isolated
vertices that are not marked are included in the same construction and
contribute their coordinate vectors.  The construction is exactly the graph
search described above, applied to the support multigraph of the given
decomposition.
\end{proof}


For each SDD-partial FR step, we represent the exposed face
using a basis matrix whose columns have pairwise disjoint supports, as given
by \Cref{prop_sdd_disjoint_support_basis}.
For the first step, let
\(W\in L^\perp\cap\FW_2^{n+1}\) be the exposing matrix, let \(V\) be this
basis of \(\ker(W)\), and parametrize the exposed face by
\(Y=VRV^T\), where \(R\in\mathbb S^r_+\).
Throughout the paper, SDD-partial FR uses at every step the
disjoint-support basis obtained by the graph-search construction in
\Cref{prop_sdd_disjoint_support_basis}.  Thus the basis is computed directly
from the factor-width-two exposing matrix as part of the procedure.
We index the \(r\) columns of \(V\), and the corresponding
rows and columns of \(R\), by \(\{0,\ldots,r-1\}\).
Although these columns have disjoint supports, their nonzero
entries need not yet be equal.  The normalization and arrow constraints will
either enforce equal weights or allow the corresponding row and column of
\(R\) to be removed at a subsequent step.

\begin{lem}\label{lemma_e0}
For any $W\in L^\perp\cap\Snop$, the vector $e_0$ does not belong to
$\range(W)$.  Consequently, if $V$ is any basis matrix for $\ker(W)$, then at
least one column $v$ of $V$ satisfies $v_0\neq0$.
\end{lem}

\begin{proof}
Choose $x\in P$.  Since $L\cap\Snop$
is an SDP relaxation for $P$, the matrix
$Y=\left(\begin{smallmatrix}1\\ x\end{smallmatrix}\right)
\left(\begin{smallmatrix}1\\ x\end{smallmatrix}\right)^T$ belongs to
$L\cap\Snop$.  Hence $\langle W,Y\rangle=0$. Since $W\succeq0$ and
$Y\succeq0$, this implies $YW=0$. If $e_0\in\range(W)$, then $e_0=Wu$ for
some vector $u$, and hence $Ye_0=YWu=0$.  Therefore
$Y_{00}=e_0^TYe_0=0$, contradicting the normalization constraint
$Y_{00}=1$ in \eqref{eq:binary_arrow_constraints}.

If every column of \(V\) had a zero zeroth entry, then
\(e_0\) would be orthogonal to \(\ker(W)\).  Since \(W\) is symmetric, this
would imply
\(e_0\in\ker(W)^\perp=\range(W)\), contradicting the first assertion.
\end{proof}

The next lemma identifies the unique basis column whose
zeroth entry is nonzero.

\begin{lem}\label{lem:zeroth_entry_column}
Let \(W\in L^\perp\cap\FW_2^{n+1}\), and let \(V\) be a basis
matrix for \(\ker(W)\) whose columns have pairwise disjoint supports, as
given by \Cref{prop_sdd_disjoint_support_basis}.
Then there is a unique column \(v\) of
\(V\) such that \(v_0\neq0\).  Moreover, \(v_i=v_0\) for every
\(i\in\supp(v)\).
\end{lem}

\begin{proof}
Disjointness of the column supports gives uniqueness, while
\Cref{lemma_e0} gives existence.

Choose $x\in P$ and set
$Y=\left(\begin{smallmatrix}1\\x\end{smallmatrix}\right)
\left(\begin{smallmatrix}1\\x\end{smallmatrix}\right)^T$.  Since $Y$ belongs
to the face exposed by $W$, write
$Y=VRV^T$ for some $R\succeq0$.
Suppose that $v_i\neq0$ for some $i\geq1$.  If \(v\) is the
column of \(V\) indexed by \(k\), then $Y_{00}=v_0^2R_{kk}=1$.
Since \(v_0\neq0\), it follows that \(R_{kk}=v_0^{-2}>0\).
The arrow constraint $Y_{ii}=Y_{0i}$ gives
$v_i^2R_{kk}=v_0v_iR_{kk}$.  Since $v_i\neq0$ and $R_{kk}\neq0$, we have
$v_i=v_0$.
\end{proof}

\begin{sloppypar}
We next determine the binary relations encoded by this
basis.
Reorder the columns of $V$ so that the unique column with a nonzero zeroth
entry is indexed by $0$.  We say that a column \(v\) has \emph{equal weights}
on its support if \(v_i=v_j\) for all \(i,j\in\supp(v)\); otherwise, it has
\emph{unequal weights} on its support.  Scale every equal-weight column so that
its nonzero entries equal one.  These operations do not change the
represented face.  In particular, $V^Te_0=e_0$.
Because the column supports are disjoint, each
\(\ell\in\{1,\ldots,n\}\) either satisfies \(V^Te_\ell=0\) or belongs to the
support of a unique column \(v\).
If \(v\) is indexed by \(k\), then
\(V^Te_i=v_i e_k\) for every \(i\in\supp(v)\).
Set
\[
\widetilde L:=\{R\in\mathbb S^r\mid VRV^T\in L\}.
\]
In the classification below, \(R\in\widetilde L\) and \(Y=VRV^T\).
For arbitrary indices \(i,j\),
\begin{equation}\label{eq:face_entry_from_basis}
Y_{ij}=e_i^TYe_j=(V^Te_i)^TR(V^Te_j).
\end{equation}
This representation yields the following three cases.
\end{sloppypar}
\begin{enumerate}[label=\emph{(\roman*)},leftmargin=*]
\item \emph{Variables fixed to zero.}  If $V^Te_\ell=0$, then every matrix
$Y=VRV^T$ on the reduced face satisfies
\begin{equation}\label{eq:sdd_fixed_zero_entries}
Y_{0\ell}=0,\qquad Y_{\ell\ell}=0.
\end{equation}
Thus the binary variable $x_\ell$ is zero for every $x\in P$, and the
current partial FR step has identified \(x_\ell\) as fixed
to zero.

\item \emph{Equal weights.}  If \(v\) has equal weights on its support, then its
nonzero entries equal one by the normalization above.  There are two subcases.

\begin{enumerate}[label=\emph{(\alph*)},leftmargin=2em]
\item \emph{Variables fixed to one.}  If $0\in\supp(v)$, then $v$ is the
column indexed by $0$.  Hence \eqref{eq:face_entry_from_basis} gives, for every
$i,j\in\supp(v)$, the identity $Y_{ij}=R_{00}=Y_{00}$.
For \(R\in\widetilde L\), the normalization
\(Y_{00}=1\) therefore makes the principal submatrix indexed by
\(\supp(v)\) the all-ones matrix.  Hence \(x_i=1\) for every
\(x\in P\) and every \(i\in\supp(v)\setminus\{0\}\).

\item \emph{Variables forced to be equal.}  If $0\notin\supp(v)$, set
$S:=\supp(v)\subseteq\{1,\ldots,n\}$.  Then
\eqref{eq:face_entry_from_basis} gives $Y_{ij}=R_{kk}$ for $i,j\in S$ and
$Y_{0i}=R_{0k}$ for $i\in S$.
The arrow constraints give
\begin{equation}\label{eq:sdd_equal_weight_reduced_arrow}
R_{kk}=Y_{ii}=Y_{0i}=R_{0k}\qquad(i\in S).
\end{equation}
Consequently, the entries in the principal submatrix on \(S\)
and the entries \(Y_{0i}\), \(i\in S\), have a common value.  In particular,
every \(x\in P\) satisfies \(x_i=x_j\) for all \(i,j\in S\).
\end{enumerate}

\begin{sloppypar}
\end{sloppypar}

\item \emph{Unequal weights.}  If \(v\) has unequal weights on its support, then
$v_i\neq v_j$ for some $i,j\in\supp(v)$.  By
\Cref{lem:zeroth_entry_column}, such a column satisfies
$0\notin\supp(v)$.
For \(R\in\widetilde L\),
\eqref{eq:face_entry_from_basis} gives
\[
Y_{ii}=v_i^2R_{kk},\qquad
Y_{0i}=v_iR_{0k},
\qquad
Y_{jj}=v_j^2R_{kk},\qquad
Y_{0j}=v_jR_{0k}.
\]
The arrow constraints \eqref{eq:binary_arrow_constraints} therefore give
$v_iR_{kk}=R_{0k}=v_jR_{kk}$.  Since $v_i\neq v_j$, we have
\begin{equation}\label{eq:sdd_unequal_weight_reduced_zero}
R_{kk}=R_{0k}=0.
\end{equation}
Consequently, for every index $q\in\supp(v)$,
$Y_{qq}=v_q^2R_{kk}$ and $Y_{0q}=v_qR_{0k}$, and hence
$Y_{0q}=Y_{qq}=0$.
Unlike case \emph{(i)}, these indices correspond to nonzero
rows of \(V\).  The next lemma shows that the reduced SDD auxiliary system
contains the diagonal exposing matrix needed to remove the \(k\)th row and
column of \(R\).
\end{enumerate}

\begin{cor}[Structure after one SDD-partial FR step]
\label{cor:sdd_one_step_structure}
Let \(V\) be a basis matrix for the kernel of the exposing
matrix, with pairwise disjoint column supports,
chosen in one SDD-partial FR step applied to an SDP satisfying
\Cref{ass:binary_sdp_relaxation}.  Its
columns can be reordered and scaled so that $V^Te_0=e_0$, and every column is
either a $\{0,1\}$ column or an unequal-weight column.  For
$\widetilde L:=\{R\in\mathbb{S}^{r}\mid VRV^T\in L\}$,
the reduced problem $\widetilde L\cap\mathbb{S}^r_+$ satisfies
\Cref{ass:binary_sdp_relaxation}.  More precisely, every
$R\in\widetilde L$ satisfies $R_{00}=1$ and, for $k=1,\ldots,r-1$,
\[
\begin{aligned}
R_{kk}&=R_{0k}
&&\text{if the $k$th column of $V$ is a $\{0,1\}$ column},\\
R_{kk}&=0,\quad R_{0k}=0
&&\text{if the $k$th column of $V$ has unequal weights}.
\end{aligned}
\]
\end{cor}
\begin{proof}
The columns of $V$ have disjoint supports by
\Cref{prop_sdd_disjoint_support_basis}.  By
\Cref{lem:zeroth_entry_column}, the unique column containing the zeroth entry
has equal nonzero entries.  Scale this column and every other equal-weight
column so that their nonzero entries equal one.

Since $V^Te_0=e_0$, \eqref{eq:face_entry_from_basis} and the normalization
constraint give $R_{00}=1$.  The remaining equalities follow from
\eqref{eq:sdd_equal_weight_reduced_arrow} for a $\{0,1\}$ column and from
\eqref{eq:sdd_unequal_weight_reduced_zero} for an unequal-weight column.
\end{proof}

For the reduced problem, the SDD-partial FR auxiliary search
is over \(\widetilde L^\perp\cap\FW_2^r\).
We retain the symmetric matrix-unit convention from
\eqref{eq:symmetric_basis_matrices}, writing $\widetilde E_{ij}$ for the
corresponding matrices in $\mathbb{S}^r$.

The next lemma shows that every unequal-weight column yields
a diagonal exposing matrix \(\widetilde E_{kk}\) in the reduced problem.
\begin{lem}[Unequal weights yield a reduced diagonal exposing matrix]\label{lem:sdd_reduced_arrow_cleanup}
Use the above ordering and scaling of the columns of \(V\), so that
\(V^Te_0=e_0\).  Let \(v\) be a column of \(V\), and let \(k\) be its column
index.
If there exist two indices $i,j\in\supp(v)$ with $v_i\neq v_j$, then
\[
\widetilde E_{kk}\in \widetilde L^\perp\cap\FW_2^r.
\]
\end{lem}
\begin{proof}
By \Cref{lem:zeroth_entry_column}, an unequal-weight column cannot contain the
index \(0\) in its support.  Hence \(i,j\in\{1,\ldots,n\}\).
By \eqref{eq:arrow_matrix_defs}, the $i$th arrow constraint has data matrix
$A_i=E_{ii}-E_{0i}$.  Since $V^Te_0=e_0$ and $V^Te_i=v_ie_k$, the
corresponding reduced data matrix is
\[
\begin{aligned}
V^TA_iV
&=
V^TE_{ii}V-V^TE_{0i}V
\\
&=
(v_ie_k)(v_ie_k)^T
-\frac{v_i}{2}\bigl(e_0e_k^T+e_ke_0^T\bigr)
\\
&=
v_i^2\widetilde E_{kk}-v_i\widetilde E_{0k}.
\end{aligned}
\]
Similarly, the $j$th arrow constraint gives
$V^TA_jV=v_j^2\widetilde E_{kk}-v_j\widetilde E_{0k}$.
Because the arrow equations have zero right-hand sides, both
reduced data matrices belong to \(\widetilde L^\perp\).
The linear combination
\[
v_jV^TA_iV-v_iV^TA_jV
=
v_iv_j(v_i-v_j)\widetilde E_{kk}
\]
cancels the off-diagonal terms.  The scalar $v_iv_j(v_i-v_j)$ is nonzero, and
$\widetilde L^\perp$ is a linear space, so
$\widetilde E_{kk}\in\widetilde L^\perp$.  Since
$\widetilde E_{kk}\in\FW_2^r$, the claim follows.
\end{proof}

The reduced problem again satisfies
\Cref{ass:binary_sdp_relaxation}, so the argument applies inductively.  At any
stage, an unequal-weight column produces a nonzero diagonal exposing matrix
by \Cref{lem:sdd_reduced_arrow_cleanup}.  Consequently, no such column can
remain when SDD-partial FR terminates.

\begin{cor}[Face identified by SDD-partial FR]
\label{cor:sdd_identified_binary_basis}
\mbox{}\\
Suppose that, at each step, SDD-partial FR uses a basis matrix
with pairwise disjoint column supports, as given by
\Cref{prop_sdd_disjoint_support_basis}.
If it terminates at the
face $F$, then $F=\{VRV^T\mid R\in\mathbb{S}^r_+\}$ for a
$\{0,1\}$ matrix $V$ with full column rank and disjoint column supports.
\end{cor}
\begin{proof}
If no reduction is performed, take $V=I$.  Otherwise, after the first step,
\Cref{cor:sdd_one_step_structure} gives the current face in the original
matrix space as $\{VRV^T\mid R\in\mathbb{S}^r_+\}$, where $V$ has disjoint
column supports.  The corresponding reduced problem
\[
\widetilde L\cap\mathbb{S}^r_+,
\qquad
\widetilde L:=\{R\in\mathbb{S}^r\mid VRV^T\in L\},
\]
again satisfies \Cref{ass:binary_sdp_relaxation}.

Apply the second SDD-partial FR step to
$\widetilde L\cap\mathbb{S}^r_+$, and let
\(U\in\mathbb R^{r\times s}\) be the basis matrix with
pairwise disjoint column supports given by
\Cref{prop_sdd_disjoint_support_basis}.
Writing
$R=U\widetilde R U^T$, the new face, viewed in the original matrix space, is
\[
\{(VU)\widetilde R(VU)^T\mid \widetilde R\in\mathbb{S}^s_+\}.
\]
If $u_\ell$ is a column of $U$, then
$\supp(Vu_\ell)=\bigcup_{k\in\supp(u_\ell)}\supp(Ve_k)$.  Since both $V$
and $U$ have disjoint column supports, so does $VU$.
Furthermore, \Cref{cor:sdd_one_step_structure}, now applied to the reduced
problem, shows that the problem obtained after the second step again satisfies
\Cref{ass:binary_sdp_relaxation}.  Replacing $V$ by $VU$ and
repeating this argument treats all subsequent steps.

At termination, the preceding iterative argument and
\Cref{lem:sdd_reduced_arrow_cleanup} show that the matrix $V$ in the
representation $F=\{VRV^T\mid R\in\mathbb S^r_+\}$ has no unequal-weight
column.  Hence every column has equal nonzero entries.
Scaling each column gives a $\{0,1\}$ matrix $V$ with disjoint column
supports, without changing the represented face $F$.
\end{proof}

The corollary gives a structural characterization of terminal SDD faces:
they record only fixed variables and groups of equal variables.  Indeed, if
\(F\) is the face identified by SDD-partial FR, then
\Cref{cor:sdd_identified_binary_basis} gives
$F=\{VRV^T\mid R\succeq0\}$ for a full-column-rank $\{0,1\}$ matrix $V$
whose columns have pairwise disjoint supports.
The space $\ker(V^T)$ is spanned by the standard unit vectors corresponding to
the zero rows of $V$ and by differences of standard unit vectors within the
support of each column.  Consequently, there exist sets
\[
I\subseteq\{1,\ldots,n\},
\qquad
M\subseteq\{(i,j)\mid 0\leq i<j\leq n\},
\]
such that
\begin{equation}\label{eq:sdd_terminal_face_equations}
F=
\left\{Y\succeq0\ \middle|\
\begin{aligned}
Ye_i&=0 &&\text{for }i\in I,\\
Y(e_i-e_j)&=0 &&\text{for }(i,j)\in M
\end{aligned}
\right\}.
\end{equation}
Compared with \eqref{eq:fw2_exposed_face_two_coordinate}, the conclusion is
stronger than two-coordinate support: at termination, the kernel vectors can
be chosen from \(e_i\) and \(e_i-e_j\), so arbitrary coefficient ratios do
not persist.
The equations indexed by \(I\) fix the corresponding
variables to zero.  A pair \((0,j)\in M\) fixes \(x_j\) to one, while a pair
\((i,j)\in M\) with \(i,j\geq1\) enforces \(x_i=x_j\).

The same basis also yields a natural reduced formulation.  Let \(S_p\) be the
support of the \(p\)th column of \(V\).  This column is the indicator vector
of \(S_p\), and the sets \(S_p\) are pairwise disjoint.
Hence, for every data matrix $B\in\mathbb{S}^{n+1}$,
\[
(V^TBV)_{pq}=\sum_{a\in S_p}\sum_{b\in S_q}B_{ab}.
\]
Thus the disjoint-support basis limits fill-in in the reduced data and can
preserve sparsity, in contrast with full FR, where a generic basis may produce
dense reduced data.

\subsection{The polyhedral case: diagonal- and DD-partial FR}
\label{sec:outer_relaxation_characterizations}

Diagonal- and DD-partial FR are special cases of SDD-partial FR, but their
polyhedral generator structure yields sharper descriptions of possible bases
for the kernel of an exposing matrix.  Under
\Cref{ass:binary_sdp_relaxation}, the two cases
specialize as follows.

\begin{enumerate}[label=\emph{(\roman*)},leftmargin=*]
\item \emph{Diagonal-partial FR.}
If diagonal-partial FR terminates at the face $F$, then
\eqref{eq:sdd_terminal_face_equations} reduces to
\[
F=\{Y\succeq0\mid Ye_i=0\ \text{for }i\in I\}
\]
for some $I\subseteq\{1,\ldots,n\}$; equivalently, $M=\varnothing$ in
\eqref{eq:sdd_terminal_face_equations}.  Indeed, an exposing matrix in $\D$ is
nonnegative diagonal, so a basis of its kernel is obtained by deleting each
standard unit column $e_i$ corresponding to a positive diagonal entry.  By
\Cref{lemma_e0}, the homogenizing column $e_0$ is never deleted.  Thus
diagonal-partial FR can only identify variables fixed to zero.

\item \emph{DD-partial FR.}
When DD-partial FR terminates, the identified face also has the form
\eqref{eq:sdd_terminal_face_equations}, by the same argument as for SDD-partial
FR.  At an intermediate step, however, the DD generator structure gives
additional information.
By \Cref{tab:polyhedral_cones}, every DD generator vector has entries in
$\{0,\pm1\}$.  The support-graph construction in
\Cref{prop_sdd_disjoint_support_basis} therefore gives a basis matrix for the
kernel whose columns have disjoint supports and entries in $\{0,\pm1\}$.  In
particular, if the nonzero entries of a column are not all equal, then the
column must contain both $+1$ and $-1$.
\end{enumerate}

The descriptions above concern the structure of the face identified at
termination.  We now consider a complementary question at an arbitrary step:
can the equations defining the next face be identified directly from the outer
relaxation?  For a polyhedral inner approximation, this identification follows
directly from Goldman--Tucker strict complementarity
\cite{goldman1956theory} and reflects the fact that facial reduction for a
polyhedral cone terminates after at most one step.  We record the result for
arbitrary nonzero positive semidefinite generators, and then apply it to $\D$ and $\DD$.

Fix nonzero matrices $G_i\in\mathbb S^N_+$ and define
\begin{equation}\label{eq:polyhedral_inner_cone}
\mathcal K
:=
\operatorname{cone}\{G_i\mid i=1,\ldots,p\}.
\end{equation}
Thus $\mathcal K$ is a polyhedral inner approximation of
$\mathbb{S}^N_+$, and
\[
\mathcal K^*
=
\{Y\in\mathbb{S}^N\mid \langle G_i,Y\rangle\geq0,
\ i=1,\ldots,p\}.
\]
Because $\mathbb{S}^N_+\subseteq\mathcal K^*$, the polyhedron
$L\cap\mathcal K^*$ is an outer relaxation of the SDP feasible set and is
nonempty.  Each generator $G_i$ defines one of its inequalities.  Let
$\mathcal J$ be the indices of the inequalities active throughout the outer
relaxation:
\begin{equation}\label{eq:active_generator_indices}
i\in\mathcal J
\quad\Longleftrightarrow\quad
\langle G_i,Y\rangle=0
\quad\text{for every }Y\in L\cap\mathcal K^*.
\end{equation}
The right-hand condition can be checked by solving
$\max\{\langle G_i,Y\rangle\mid Y\in L\cap\mathcal K^*\}$: it holds exactly
when the optimal value is zero, with the value taken as $+\infty$ if the LP is
unbounded.

\begin{prop}[The polyhedral outer relaxation determines the next face]
\label{prop:polyhedral_exact_characterization}
If $W\in\ri(L^\perp\cap\mathcal K)$, then
\begin{equation}\label{eq:polyhedral_active_face}
\mathbb S^N_+\cap W^\perp
=
\{Y\succeq0\mid YG_i=0
  \text{ for every }i\in\mathcal J\}.
\end{equation}
\end{prop}

\begin{proof}
Define
\[
Q:=\left\{\lambda\in\mathbb R_+^p\ \middle|\
\sum_{i=1}^p\lambda_iG_i\in L^\perp\right\},
\qquad
T(\lambda):=\sum_{i=1}^p\lambda_iG_i.
\]
Then $T(Q)=L^\perp\cap\mathcal K$.  For any $\lambda\in Q$ and
$Y\in L\cap\mathcal K^*$,
\[
0=\left\langle\sum_{i=1}^p\lambda_iG_i,Y\right\rangle
=\sum_{i=1}^p\lambda_i\langle G_i,Y\rangle.
\]
All terms are nonnegative, so $\lambda_i>0$ implies $i\in\mathcal J$.
Conversely, the Goldman--Tucker strict-complementarity theorem
\cite{goldman1956theory}, applied to the LP feasibility system
$L\cap\mathcal K^*$, gives some $\widehat\lambda\in Q$ satisfying
$\widehat\lambda_i>0$ for every $i\in\mathcal J$.  It follows that
the coordinates of $Q$ that can be positive are precisely those indexed by
$\mathcal J$.  Since $Q$ is the intersection of a linear subspace with
$\mathbb R_+^p$,
\[
\ri(Q)=\{\lambda\in Q\mid
\lambda_i>0\ \text{for every }i\in\mathcal J\}.
\]
Since linear maps preserve relative interiors,
$T(\ri(Q))=\ri(L^\perp\cap\mathcal K)$.  Hence
$W=\sum_{i\in\mathcal J}\lambda_i G_i$ for some $\lambda_i>0$.
For $Y\succeq0$,
\[
\langle W,Y\rangle
=\sum_{i\in\mathcal J}\lambda_i\langle G_i,Y\rangle.
\]
The summands are nonnegative, and, because \(G_i,Y\succeq0\),
\(\langle G_i,Y\rangle=0\) is equivalent to \(YG_i=0\).
This proves \eqref{eq:polyhedral_active_face}.
\end{proof}

Applying the proposition with the generators in \Cref{tab:polyhedral_cones}
makes this identification explicit.  For $\D$, the next face is obtained by
imposing $Ye_k=0$ precisely for the inequalities associated with $e_ke_k^T$
that are active throughout the outer relaxation.  For $\DD$, one additionally
imposes $Y(e_k\pm e_\ell)=0$ precisely for the active inequalities associated
with $(e_k\pm e_\ell)(e_k\pm e_\ell)^T$.  Thus, the preceding LP tests provide
a direct characterization of the next face.

\section{Maximum partial FR length with singularity degree two}
\label{sec:sharp_partial_fra_separation}

Sturm's classical SDP family has a matrix variable of order $n+1$ and
singularity degree $n$ \cite{sturm2000error}.  Its exposing matrices can be
chosen from $\FW_2^{n+1}$, so restricting the exposing-matrix search to
$\FW_2^{n+1}$ does not increase the number of steps on that family:
SDD-partial FR also requires $n$ steps.  However, Sturm's formulation does not
include the normalization constraint and arrow constraints in
\Cref{ass:binary_sdp_relaxation}.  The family constructed below satisfies
this assumption and has singularity degree two,
yet both DD- and SDD-partial FR require
$n=(n+1)-1$ steps, the maximum possible for a nonzero feasible SDP over
$\mathbb S^{n+1}_+$.

At every partial FR step, the exposing matrix is chosen from
$\ri(L^\perp\cap\mathcal K)$, where $L$ is the affine subspace in
the current reduced formulation and $\mathcal K$ is the cone used by the
partial FR auxiliary system.
We use the basis matrices displayed below, whose columns have pairwise
disjoint supports, to parametrize the resulting faces.
The displayed matrices are precisely the bases produced by
the graph-search construction in
\Cref{prop_sdd_disjoint_support_basis}.

Let $m\geq1$ and $n=2m$.  Define vectors in
$\mathbb{R}^{n+1}$ by
\[
q_1=e_1+e_2,
\qquad
q_j=e_{2j-2}+e_{2j-1}+e_{2j}\quad(j=2,\ldots,m).
\]
The binary set
\[
P_m
:=
\{x\in\{0,1\}^{2m}\mid
q_j^T\binom{0}{x}=0,\
j=1,\ldots,m\}
\]
is $\{0\}$: the first equation gives $x_1=x_2=0$, and the
remaining equations then force the other variables to zero
in pairs.
Define the affine subspace
\[
L_m
:=
\left\{
Y\in\mathbb{S}^{n+1}
\ \middle|\
\begin{aligned}
Y_{00}&=1,\\
Y_{ii}&=Y_{0i} && (i=1,\ldots,n),\\
\langle q_jq_j^T,Y\rangle&=0 && (j=1,\ldots,m)
\end{aligned}
\right\}.
\]
Thus $L_m\cap\mathbb{S}^{n+1}_+$ is an SDP relaxation of $P_m$ satisfying
\Cref{ass:binary_sdp_relaxation}.

\begin{lem}[Singularity degree]
\label{lem:sharp_family_full_fr}
The SDP relaxation
$L_m\cap\mathbb{S}^{n+1}_+$ has singularity degree two.
\end{lem}
\begin{proof}[Proof]
The rank-one lift $e_0e_0^T$ of $x=0$ belongs to
$L_m\cap\mathbb S^{n+1}_+$.  Hence every exposing matrix
$W\in L_m^\perp\cap\mathbb S^{n+1}_+$ satisfies
$0=\langle W,e_0e_0^T\rangle=W_{00}$, so its zeroth row and column vanish.
In any representation of \(W\) by the constraint matrices, its off-diagonal entries
in the zeroth row first force all arrow-constraint multipliers to vanish; the
\((0,0)\) entry then forces the normalization-constraint multiplier to
vanish.  Moreover,
for $j,\ell\in\{1,\ldots,m\}$, the $(2j-1)$st component of
$q_\ell$ is nonzero if and only if $\ell=j$.  Hence
$q_1,\ldots,q_m$ are linearly independent.
Writing \(Q=(q_1\ \cdots\ q_m)\), a left inverse of \(Q\)
shows that \(Q\operatorname{Diag}(\alpha)Q^T\succeq0\) holds if and only if
\(\alpha\geq0\).  Consequently,
\begin{equation}\label{eq:sharp_family_full_fr_exposing_matrices}
L_m^\perp\cap\mathbb{S}^{n+1}_+
=
\left\{
\sum_{j=1}^m\alpha_jq_jq_j^T
\ \middle|\
\alpha_j\geq0
\right\}.
\end{equation}
The matrix
$W_1:=\sum_{j=1}^m q_jq_j^T$ belongs to
$\ri(L_m^\perp\cap\mathbb S^{n+1}_+)$ and has rank $m$.
Let $\bar q_j\in\mathbb{R}^n$ consist of the last $n$ entries of
$q_j$, and let
$H:=\{u\in\mathbb{R}^{n}\mid \bar q_j^Tu=0,\ j=1,\ldots,m\}$.
Let the columns of $U\in\mathbb{R}^{n\times m}$ form a basis
of $H$.  After the first facial-reduction step, let
\[
V:=
\begin{pmatrix}
1 & 0\\
0 & U
\end{pmatrix},
\qquad
\lambda:=\sum_{j=1}^m\bar q_j
=\mathbf{1}+\sum_{j=1}^{m-1}e_{2j}.
\]
Every entry of $\lambda\in\mathbb{R}^n$ is positive, and
$U^T\lambda=0$.  Recall from \eqref{eq:arrow_matrix_defs} that
$A_i=E_{ii}-E_{0i}$ is the data matrix of the $i$th arrow constraint.
The weighted sum of the reduced arrow-constraint matrices is
\[
\sum_{i=1}^{n}\lambda_i V^TA_iV
=
\begin{pmatrix}
0 & 0\\
0 & U^T\operatorname{Diag}(\lambda)U
\end{pmatrix}.
\]
Since \(U\) has full column rank and
\(\operatorname{Diag}(\lambda)\succ0\), its lower-right block is positive
definite.  Thus, the second facial-reduction
step exposes the ray \(\{t e_0e_0^T\mid t\geq0\}\).  The normalization
constraint fixes \(t=1\); hence \(e_0e_0^T\) is a relative-interior feasible
point and the reduced problem satisfies Slater's condition.  A single
facial-reduction step cannot suffice: every
$W\in L_m^\perp\cap\mathbb S^{n+1}_+$ has rank at most
$m<n=2m$, whereas exposing this minimal face in a single
facial-reduction step would require rank
$n$.  Hence the singularity degree is exactly two.
\end{proof}

\begin{thm}[Maximum SDD-partial FR length]
\label{thm:sharp_n_step_partial_fra}
SDD-partial FR requires $2m=n$ steps on
$L_m\cap\mathbb S^{n+1}_+$.
\end{thm}

\begin{proof}
By \eqref{eq:sharp_family_full_fr_exposing_matrices}, every
$W\in L_m^\perp\cap\mathbb S^{n+1}_+$ has the form
$W=\sum_{j=1}^m\alpha_jq_jq_j^T$ with $\alpha_j\geq0$.
Let $z:=\sum_{i=1}^n e_i\in\mathbb R^{n+1}$.  Since every
$q_jq_j^T$ is entrywise nonnegative and \(q_1\) and \(q_j\) for \(j\geq2\)
have support sizes two and three,
respectively, \eqref{eq:comparison_matrix_def} gives
$z^TC(W)z=-3\sum_{j=2}^m\alpha_j$.  Hence
$C(W)\succeq0$ forces
$\alpha_2=\cdots=\alpha_m=0$.  Thus,
$L_m^\perp\cap\FW_2^{n+1}=\{\gamma q_1q_1^T\mid\gamma\geq0\}$, whose
relative interior consists of the positive multiples of $q_1q_1^T$.
Therefore the first SDD-partial FR step chooses a positive multiple of
$q_1q_1^T$ as its exposing matrix and exposes the face
$\mathbb S^{n+1}_+\cap(q_1q_1^T)^\perp$.

Use the following basis matrix for $\ker(q_1q_1^T)$:
\[
V_1=\begin{pmatrix}e_0 & e_2-e_1 & e_3 & \cdots & e_{2m}\end{pmatrix}
\in\mathbb{R}^{(2m+1)\times 2m}.
\]
The resulting reduced SDP is
\[
\widetilde L_1\cap\mathbb{S}^{2m}_+,
\qquad
\widetilde L_1
:=
\{R\in\mathbb{S}^{2m}\mid V_1RV_1^T\in L_m\}.
\]
Index the columns of \(V_1\), in the order displayed above, by
\(\{0,\ldots,2m-1\}\), and use the same indexing for the standard basis
\(\widetilde e_0,\ldots,\widetilde e_{2m-1}\) of the reduced space
\(\mathbb R^{2m}\).  Thus, columns \(0\) and \(1\) of \(V_1\) are \(e_0\) and
\(e_2-e_1\), respectively.
By \eqref{eq:arrow_matrix_defs}, the first two reduced arrow constraints satisfy
$V_1^T(A_1+A_2)V_1=2\widetilde e_1\widetilde e_1^T$, so
$\widetilde e_1\widetilde e_1^T$ is a diagonal exposing matrix
for the reduced SDP.  The remaining reduced data matrices are
$\widehat q_j\widehat q_j^T$, $j=2,\ldots,m$, where
\[
\widehat q_2=\widetilde e_1+\widetilde e_2+\widetilde e_3,
\qquad
\widehat q_j
=\widetilde e_{2j-3}+\widetilde e_{2j-2}+\widetilde e_{2j-1}
\quad(j\geq3).
\]
The same zeroth-row argument as above shows that every
$W\in\widetilde L_1^\perp\cap\mathbb{S}^{2m}_+$ has the form
\[
W=\alpha\widetilde e_1\widetilde e_1^T+
\sum_{j=2}^m\beta_j\widehat q_j\widehat q_j^T,
\qquad \alpha,\beta_j\geq0.
\]
The vectors
$\widetilde e_1,\widehat q_2,\ldots,\widehat q_m$ are linearly independent, so
the same full-column-rank argument used in the proof of
\mbox{\Cref{lem:sharp_family_full_fr}} shows that
positive semidefiniteness forces all displayed coefficients to be
nonnegative.
Define $z\in\mathbb{R}^{2m}$ by
$z_0=z_1=0$ and $z_i=1$ for $i=2,\ldots,2m-1$.
The diagonal term
\(\alpha\widetilde e_1\widetilde e_1^T\) contributes zero to \(z^TC(W)z\).
On the support of \(\widehat q_2\), the entries of \(z\) are \(0,1,1\);
hence \(\beta_2\widehat q_2\widehat q_2^T\) contributes
\(\beta_2(2-2)=0\).
For each $j\geq3$, the three entries of $z$ indexed by
$\supp(\widehat q_j)$ are equal to one, and hence
\[
z^TC(\widehat q_j\widehat q_j^T)z=3-6=-3.
\]
Moreover, all summands in $W$ are entrywise nonnegative, so there is no
cancellation inside the absolute values in
\eqref{eq:comparison_matrix_def}.  Therefore
$z^TC(W)z=-3\sum_{j=3}^m\beta_j$, so $C(W)\succeq0$ forces
$\beta_3=\cdots=\beta_m=0$.  If $\beta_2>0$, then the comparison matrix of the
remaining block on the coordinates
$\widetilde e_1,\widetilde e_2,\widetilde e_3$ is
\[
\begin{pmatrix}
\alpha+\beta_2 & -\beta_2 & -\beta_2\\
-\beta_2 & \beta_2 & -\beta_2\\
-\beta_2 & -\beta_2 & \beta_2
\end{pmatrix}.
\]
The cofactor of the upper-left entry is zero, so the
determinant is independent of $\alpha$; direct evaluation gives
$-4\beta_2^3<0$.  This contradicts positive semidefiniteness.  Thus
$\beta_2=0$, and
$\widetilde L_1^\perp\cap\FW_2^{2m}$ is the ray generated by
$\widetilde e_1\widetilde e_1^T$.
Therefore, the second SDD-partial FR step chooses
$\widetilde e_1\widetilde e_1^T$ as its exposing matrix.  The resulting face is
\[
\mathbb S^{2m}_+\cap
(\widetilde e_1\widetilde e_1^T)^\perp
=
\{R\in\mathbb S^{2m}_+\mid R\widetilde e_1=0\};
\]
equivalently, the row and column of $R$ indexed by $1$ are zero.

After these two steps, let
$V_2=\begin{pmatrix}e_0 & e_3 & \cdots & e_{2m}\end{pmatrix}
\in\mathbb R^{(2m+1)\times(2m-1)}$.  The current face, expressed in the
original matrix space, is
\[
\{V_2SV_2^T\mid S\in\mathbb S^{2m-1}_+\}.
\]
Index the columns of \(V_2\), in the order displayed above, by
\(\{0,\ldots,2m-2\}\), and use the same indexing for the standard basis
\(\bar e_0,\ldots,\bar e_{2m-2}\) of the reduced space
\(\mathbb R^{2m-1}\).  Thus, column \(0\) of \(V_2\) is \(e_0\), while column
\(i\) is \(e_{i+2}\) for \(i\geq1\).  Direct calculation gives
\[
V_2^Tq_2=\bar e_1+\bar e_2,
\qquad
V_2^Tq_j
=\bar e_{2j-4}+\bar e_{2j-3}+\bar e_{2j-2}
\quad(j=3,\ldots,m).
\]
After identifying $\bar e_i$ with $e_i$, these are precisely
$q_1,\ldots,q_{m-1}$ for the instance $L_{m-1}$; the normalization and
remaining arrow constraints also have the same form.
Thus, the SDP obtained after the first two SDD-partial FR steps is
\[
\left\{
S\in\mathbb S^{2m-1}_+
\ \middle|\
V_2SV_2^T\in L_m
\right\}
=
L_{m-1}\cap\mathbb S^{2(m-1)+1}_+.
\]
Let $T_m$ denote the number of SDD-partial FR steps on
$L_m\cap\mathbb S^{2m+1}_+$.
For \(m=1\),
the first step exposes
$(e_1+e_2)(e_1+e_2)^T$ and the second
removes the row and column associated with \(e_2-e_1\),
so $T_1=2$.  The recurrence obtained above therefore gives
\[
T_m=2+T_{m-1},
\qquad T_1=2,
\qquad T_m=2m=n.
\]

By the standard singularity-degree upper bound for a
nonzero feasible SDP over \(\mathbb S^{n+1}_+\), a facial-reduction sequence in
which every step exposes a proper face has at most \(n\) steps.  Hence this
family attains the bound.
\end{proof}

To illustrate the recursive reduction used in the proof,
consider \(m=2\): each SDD-partial FR step removes one row and column, giving
the following four-step sequence.
\begin{enumerate}[leftmargin=2em]
\item \emph{First step: $Y\in\mathbb S^5_+$.}
The data matrices associated with the two lifted quadratic equations are
\[
q_1q_1^T
=
\begin{pmatrix}
0&0&0&0&0\\
0&1&1&0&0\\
0&1&1&0&0\\
0&0&0&0&0\\
0&0&0&0&0
\end{pmatrix},
\qquad
q_2q_2^T
=
\begin{pmatrix}
0&0&0&0&0\\
0&0&0&0&0\\
0&0&1&1&1\\
0&0&1&1&1\\
0&0&1&1&1
\end{pmatrix}.
\]
The first step chooses $q_1q_1^T$ as its exposing matrix.  A matrix whose
columns form a basis of its kernel is
\[
V=
\begin{pmatrix}
1& 0&0&0\\
0&-1&0&0\\
0& 1&0&0\\
0& 0&1&0\\
0& 0&0&1
\end{pmatrix}
\in\mathbb R^{5\times4}.
\]
Thus $Y=VRV^T$ with $R\in\mathbb S^4_+$.

\item \emph{Second step: $R\in\mathbb S^4_+$.}
The transformed quadratic data matrices are
\[
V^Tq_1q_1^TV=0,
\qquad
V^Tq_2q_2^TV
=
\begin{pmatrix}
0&0&0&0\\
0&1&1&1\\
0&1&1&1\\
0&1&1&1
\end{pmatrix}.
\]
The reduced arrow constraints yield the exposing matrix
\[
\begin{pmatrix}
0&0&0&0\\
0&1&0&0\\
0&0&0&0\\
0&0&0&0
\end{pmatrix}
\in\mathbb S^4_+.
\]
This exposing matrix forces the row and column of \(R\) indexed by \(1\) to
vanish.  Deleting them leaves a matrix variable \(S\in\mathbb S^3_+\).

\item \emph{Third step: $S\in\mathbb S^3_+$.}
The remaining quadratic data matrix is
\[
\begin{pmatrix}
0&0&0\\
0&1&1\\
0&1&1
\end{pmatrix}.
\]
The third step chooses this matrix as its exposing matrix.  A matrix whose
columns form a basis of its kernel is
\[
U=
\begin{pmatrix}
1& 0\\
0&-1\\
0& 1
\end{pmatrix}
\in\mathbb R^{3\times2}.
\]
Thus $S=UTU^T$ with $T\in\mathbb S^2_+$.

\item \emph{Fourth step: $T\in\mathbb S^2_+$.}
The reduced arrow constraints yield the exposing matrix
\[
\begin{pmatrix}
0&0\\
0&1
\end{pmatrix}.
\]
This exposing matrix forces the row and column of \(T\) indexed by \(1\) to
vanish.  Hence the exposed face is the ray
\[
\left\{
t\begin{pmatrix}
1&0\\
0&0
\end{pmatrix}
\ \middle|\ t\geq0
\right\}.
\]
The reduced normalization constraint is \(T_{00}=1\), so \(t=1\) is a
relative-interior feasible point of this ray and the reduction terminates.
\end{enumerate}

\begin{cor}[Maximum DD-partial FR length]
\label{cor:sharp_n_step_dd_partial_fra}
DD-partial FR requires $2m=n$ steps on
$L_m\cap\mathbb S^{n+1}_+$.
\end{cor}
\begin{proof}
At each stage of the proof of
\Cref{thm:sharp_n_step_partial_fra}, write $L$ for the affine subspace in the
current reduced formulation.  The cone $L^\perp\cap\FW_2$ is a ray generated
either by a matrix of the form $(e_i+e_j)(e_i+e_j)^T$ or by a nonnegative
diagonal matrix.  In both cases, the generator belongs to $\DD$.
Together with $\DD\subseteq\FW_2$, it follows that
\[
L^\perp\cap\DD=L^\perp\cap\FW_2
\]
at every stage.
Hence, DD-partial FR exposes the same sequence of faces and also requires
$2m=n$ steps.
\end{proof}

\begin{remark}[Dependence on the affine representation]
The relaxation \(L_m\) does not explicitly impose the lifted linear equations
\(q_j^TYe_0=0\) corresponding to \(q_j^Tx=0\).  These equations are redundant
over the PSD cone because
\(\langle q_jq_j^T,Y\rangle=0\) and \(Y\succeq0\) imply \(Yq_j=0\).
Their inclusion would nevertheless change the partial FR auxiliary system.
Indeed, with
\(H_j:=\frac12(e_0q_j^T+q_je_0^T)\), the arrow matrices satisfy
\[
H_j+\sum_{i\in\supp(q_j)}A_i
=\sum_{i\in\supp(q_j)}E_{ii}.
\]
Although \(H_j\) itself need not be positive semidefinite, it
is the data matrix of the added equality, and the combined matrix on the
right-hand side belongs to \(\D\subseteq\DD\subseteq\FW_2\).
Since the supports of \(q_1,\ldots,q_m\) cover all variable indices, summing
these identities gives
\[
W_{\mathrm{diag}}
:=\sum_{j=1}^m\left(
H_j+\sum_{i\in\supp(q_j)}A_i
\right)
=\sum_{i=1}^n\lambda_iE_{ii}
=
\begin{pmatrix}
0&0\\
0&\operatorname{Diag}(\lambda)
\end{pmatrix},
\qquad
\lambda=\mathbf{1}+\sum_{j=1}^{m-1}e_{2j}\in\mathbb R_{++}^n.
\]
Thus \(W_{\mathrm{diag}}\in\D\subseteq\DD\subseteq\FW_2\) and
\(\ker(W_{\mathrm{diag}})=\operatorname{span}\{e_0\}\), so it exposes the
minimal PSD face in one step.  Such dependence is intrinsic to facial
reduction: singularity degree is a property of the affine conic system, not of
its feasible set alone.  Accordingly, the \(n\)-step lower bound pertains to
the representation \(L_m\) defined above and does not persist after the
redundant lifted linear equations are added.
\end{remark}

\section[A support limitation of factor-width searches]
{A support limitation of factor-width searches}
\label{sec:factor_width_support_limitation}
The disjoint-support conclusion
of \Cref{prop_sdd_disjoint_support_basis} is specific to factor width
two.  Indeed, let \(k\geq3\) and let \(w\in\mathbb R^k\) have full support.
Then \(ww^T\in\FW_k^k\), and \(\ker(ww^T)\) has dimension \(k-1\).  Every
nonzero vector in this kernel has support of cardinality at least two, so a
basis cannot have pairwise disjoint supports because \(2(k-1)>k\).  A stronger
limitation is that a factor-width restriction can prevent partial FR from
detecting a valid affine equality involving all variables.

Consider
\[
P=\left\{x\in\{0,1\}^{n}\ \middle|\ \sum_{i=1}^{n}x_i=1\right\}.
\]
Let $L\cap\Snop$ be an SDP relaxation satisfying
\Cref{ass:binary_sdp_relaxation} and containing the lifted equality
\[
\langle uu^T,Y\rangle=0,
\qquad
u=-e_0+\sum_{i=1}^{n}e_i.
\]
Then
\[
uu^T\in L^\perp\cap\Snop,
\qquad\text{but}\qquad
L^\perp\cap\FW_k^{n+1}=\{0\}
\quad\text{for every }k\leq n.
\]
To prove the second assertion, fix
$W\in L^\perp\cap\FW_k^{n+1}$ with $k\leq n$.  By the definition of factor
width, write $W=\sum_{r=1}^s v^r(v^r)^T$, where
$|\supp(v^r)|\leq k$.
For each \(i=1,\ldots,n\), the matrix
\((e_0+e_i)(e_0+e_i)^T\) is the rank-one lift of a point in \(P\) and hence
belongs to \(L\).  Therefore
\[
0
=
\langle W,(e_0+e_i)(e_0+e_i)^T\rangle
=
\sum_{r=1}^s\bigl((v^r)_0+(v^r)_i\bigr)^2.
\]
Thus $(v^r)_i=-(v^r)_0$ for every $i$ and $r$, so each $v^r$ is a multiple of
$u=(-1,\mathbf{1})^T$.  Any nonzero multiple has support $n+1$, contradicting
$|\supp(v^r)|\leq k\leq n$.  Hence every $v^r$ is zero and therefore $W=0$.

Consequently, \(\FW_k\)-partial FR stops immediately for every \(k\leq n\),
although
$uu^T\in L^\perp\cap\mathbb S^{n+1}_+$ is a dense exposing matrix
available to full facial reduction.
The obstruction is the support size: every nonzero affine equality valid on
\(P\) is a nonzero multiple of \(\sum_i x_i=1\), whose homogenized coefficient
vector has support \(n+1\).

\paragraph{Comparison with affine facial reduction.}\hspace{0pt}
Affine facial reduction \cite{hu2023affine} instead constructs exposing
matrices from equations defining \(\aff(P)\), without imposing a factor-width
bound.
For the present set,
$\aff(P)=\{x\in\mathbb{R}^n\mid\sum_i x_i=1\}$, so affine facial reduction
constructs $uu^T$ directly and exposes
$F:=\{Y\succeq0\mid Yu=0\}$.  The average of the rank-one lifts of the points
in $P$ is
\[
\bar Y=\frac1n
\begin{bmatrix}
n & \mathbf{1}^T\\
\mathbf{1} & I_n
\end{bmatrix}.
\]
Because
\(L\) is affine and contains every rank-one lift, \(\bar Y\in L\).  Moreover,
\(\ker(\bar Y)=\operatorname{span}\{u\}\), and hence
\(\bar Y\in L\cap\ri(F)\).  Therefore \(F\) is the minimal face containing the
feasible set, and the resulting reduced formulation satisfies Slater's
condition.

\section{Conclusion}

For semidefinite relaxations of binary sets satisfying the
normalization and arrow constraints, the face identified by SDD-partial FR
admits a full-column-rank \(\{0,1\}\) basis matrix with pairwise disjoint
column supports.  Equivalently, this face records only variables fixed to zero
or one and groups of variables forced to be equal.  The basis yields a natural
reduced formulation in which fixed variables are eliminated, equality groups
are aggregated, and sparse data can remain sparse.  Diagonal-partial FR
deletes only indices corresponding to variables fixed to zero, while
the exposing matrices used by DD-partial FR have kernels admitting
disjoint-support bases with entries in \(\{0,\pm1\}\).

For a polyhedral inner approximation, the inequalities active throughout the associated LP
outer relaxation determine the next PSD face exactly.  In particular, for
diagonal- and DD-partial FR, the LP tests with optimal value zero identify the
active generator directions, whose common null space determines the face
exposed at that step.

The two constructions reveal complementary limitations of the
restricted search.  For the affine conic system \(L_m\cap\mathbb S^{n+1}_+\),
full facial reduction has singularity degree two, whereas DD- and SDD-partial
FR require \(n\) steps, the maximum
possible for a nonzero feasible SDP over \(\mathbb S^{n+1}_+\).
Thus tractable auxiliary problems at individual steps do not
guarantee a short reduction sequence.  In the sum-to-one example,
\(\FW_k\)-partial FR stops immediately for every \(k\leq n\): the homogenized
coefficient vector of every nonzero valid affine equation has support
\(n+1\), so the dense exposing matrix is unavailable to the restricted
search.  It is nevertheless available to full facial reduction and is
constructed directly by affine facial reduction.

Together, these results make explicit the tradeoff in partial
FR based on tractable inner approximations.  LP and SOCP auxiliary systems can
identify interpretable faces and yield sparse reduced formulations, but they
may require the maximum number of reduction steps or terminate before reaching
the minimal face.  The choice of inner approximation therefore governs not
only the cost of each auxiliary problem, but also the relations that can be
detected and the length of the resulting reduction sequence.

\bibliographystyle{siam}
\bibliography{mybib}

\begin{thebibliography}{10}

\bibitem{ahmadi2017optimization}
{\sc A.~A. Ahmadi and A.~Majumdar}, {\em Optimization over polynomials with
  sparsity-promoting constraints}, IEEE Transactions on Automatic Control, 62
  (2017), pp.~4967--4982.

\bibitem{ahmadi2019dsos}
\leavevmode\vrule height 2pt depth -1.6pt width 23pt, {\em {DSOS} and {SDSOS}
  optimization: More tractable alternatives to sum of squares and semidefinite
  optimization}, SIAM Journal on Applied Algebra and Geometry, 3 (2019),
  pp.~193--230.

\bibitem{blekherman2022sparse}
{\sc G.~Blekherman, H.~T. Ha, and T.~Riener}, {\em Sparse Polynomial
  Optimization}, Springer, 2022.

\bibitem{boman2005factor}
{\sc E.~G. Boman, D.~Chen, O.~Parekh, and S.~Toledo}, {\em On factor width and
  symmetric {H}-matrices}, Linear Algebra and its Applications, 405 (2005),
  pp.~239--248.

\bibitem{borwein1981regularizing}
{\sc J.~Borwein and H.~Wolkowicz}, {\em Regularizing the abstract convex
  program}, Journal of Mathematical Analysis and Applications, 83 (1981),
  pp.~495--530.

\bibitem{borwein1981facial}
{\sc J.~M. Borwein and H.~Wolkowicz}, {\em Facial reduction for a cone-convex
  programming problem}, Journal of the Australian Mathematical Society, 30
  (1981), pp.~369--380.

\bibitem{drusvyatskiy2017many}
{\sc D.~Drusvyatskiy and H.~Wolkowicz}, {\em The many faces of degeneracy in
  conic optimization}, Foundations and Trends{\textregistered} in Optimization,
  3 (2017), pp.~77--170.

\bibitem{goemans1995improved}
{\sc M.~X. Goemans and D.~P. Williamson}, {\em Improved approximation
  algorithms for maximum cut and satisfiability problems using semidefinite
  programming}, Journal of the ACM (JACM), 42 (1995), pp.~1115--1145.

\bibitem{goldman1956theory}
{\sc A.~J. Goldman and A.~W. Tucker}, {\em Theory of linear programming}, in
  Linear Inequalities and Related Systems, H.~W. Kuhn and A.~W. Tucker, eds.,
  vol.~38 of Annals of Mathematics Studies, Princeton University Press,
  Princeton, NJ, 1956, pp.~53--98.

\bibitem{hu2026shor}
{\sc H.~Hu}, {\em On the singularity degree of {Shor} relaxations for 0--1
  programs}.
\newblock Preprint, arXiv:2607.12476, 2026.

\bibitem{hu2025primal}
{\sc H.~Hu and M.~Xu}, {\em A primal approach to facial reduction for {SDP}
  relaxations of combinatorial optimization problems}.
\newblock Preprint,
  \url{https://optimization-online.org/wp-content/uploads/2025/07/primalFRA.pdf},
  2025.

\bibitem{hu2023affine}
{\sc H.~Hu and B.~Yang}, {\em Affine {FR} : an effective facial reduction
  algorithm for semidefinite relaxations of combinatorial problems}.
\newblock Preprint,
  \url{https://optimization-online.org/wp-content/uploads/2023/09/autoFR-1.pdf},
  2023.

\bibitem{kirschner2024predictor}
{\sc F.~Kirschner and E.~de~Klerk}, {\em A predictor-corrector algorithm for
  semidefinite programming that uses the factor width cone}, Vietnam Journal of
  Mathematics, 53 (2025), pp.~495--515.

\bibitem{laurent2003comparison}
{\sc M.~Laurent}, {\em A comparison of the {Sherali--Adams},
  {Lov{\'a}sz--Schrijver}, and {Lasserre} relaxations for 0--1 programming},
  Mathematics of Operations Research, 28 (2003), pp.~470--496.

\bibitem{lourenco2015solving}
{\sc B.~F. Louren{\c{c}}o, M.~Muramatsu, and T.~Tsuchiya}, {\em Solving {SDP}
  completely with an interior point oracle}, Optimization Methods and Software,
  36 (2021), pp.~425--471.

\bibitem{lovasz1991cones}
{\sc L.~Lov{\'a}sz and A.~Schrijver}, {\em Cones of matrices and set-functions
  and 0--1 optimization}, SIAM Journal on Optimization, 1 (1991), pp.~166--190.

\bibitem{permenter2018simplified}
{\sc F.~Permenter and P.~Parrilo}, {\em Partial facial reduction: simplified,
  equivalent {SDP}s via approximations of the {PSD} cone}, Mathematical
  Programming, 171 (2018), pp.~1--54.

\bibitem{roig2022globally}
{\sc B.~Roig-Solvas and M.~Sznaier}, {\em A globally convergent {LP} and
  {SOCP}-based algorithm for semidefinite programming}, arXiv preprint
  arXiv:2202.12374,  (2022).

\bibitem{sherali1990hierarchy}
{\sc H.~D. Sherali and W.~P. Adams}, {\em A hierarchy of relaxations between
  the continuous and convex hull relaxations for 0--1 programming}, SIAM
  Journal on Discrete Mathematics, 3 (1990), pp.~411--430.

\bibitem{sherali2013reformulation}
\leavevmode\vrule height 2pt depth -1.6pt width 23pt, {\em A
  Reformulation-Linearization Technique for Solving Discrete and Continuous
  Nonconvex Problems}, Springer Science \& Business Media, 2013.

\bibitem{shor1987quadratic}
{\sc N.~Z. Shor}, {\em Quadratic optimization problems}, Soviet Journal of
  Computer and Systems Sciences, 25 (1987), pp.~1--11.

\bibitem{sootla2019block}
{\sc A.~Sootla, Y.~Zheng, and A.~Papachristodoulou}, {\em Block
  factor-width-two matrices in semidefinite programming}, in 2019 18th European
  Control Conference (ECC), IEEE, 2019, pp.~1981--1986.

\bibitem{sremac2021error}
{\sc S.~Sremac, H.~J. Woerdeman, and H.~Wolkowicz}, {\em Error bounds and
  singularity degree in semidefinite programming}, SIAM Journal on
  Optimization, 31 (2021), pp.~812--836.

\bibitem{sturm2000error}
{\sc J.~F. Sturm}, {\em Error bounds for linear matrix inequalities}, SIAM
  Journal on Optimization, 10 (2000), pp.~1228--1248.

\bibitem{wang2021polyhedral}
{\sc Y.~Wang and S.~Burer}, {\em Polyhedral approximations in conic
  optimization}, Mathematical Programming, 188 (2021), pp.~477--504.

\end{thebibliography}
\addcontentsline{toc}{section}{Bibliography}

\end{document}